\documentclass[]{amsart}

\usepackage{amssymb}
\usepackage{enumerate}
\usepackage{comment}
\usepackage[all]{xy}
\usepackage{bm} 

\usepackage[french,english]{babel}

\newtheorem{proposition}{Proposition}
\newtheorem{theorem}[proposition]{Theorem}
\newtheorem{lemma}[proposition]{Lemma}
\newtheorem{corollary}[proposition]{Corollary}
\newtheorem{observation}[proposition]{Observation}

\theoremstyle{definition}  
\newtheorem{axioms}[proposition]{Axioms} 
\newtheorem{definition}[proposition]{Definition}
\newtheorem{example}[proposition]{Example}
\newtheorem{remark}[proposition]{Remark}

\newcommand{\secref}[1]{Section~\ref{#1}}
\newcommand{\thmref}[1]{Theorem~\ref{#1}}
\newcommand{\propref}[1]{Proposition~\ref{#1}}

\newcommand{\lemref}[1]{Lemma~\ref{#1}}
\newcommand{\cororef}[1]{Corollary~\ref{#1}}
\newcommand{\defref}[1]{Definition~\ref{#1}}
\newcommand{\remref}[1]{Remark~\ref{#1}}
\newcommand{\obsref}[1]{Observation~\ref{#1}}
\newcommand{\examref}[1]{Example~\ref{#1}}

\newcommand{\petitespace}{\nobreak\hspace{0.15em}}
\newcommand{\negskip}{\vspace{-1ex}}
\newdir{c}{{}*!/-5pt/@{>}*{}}  

\newcommand{\Jc}{\mathcal{T}}
\newcommand{\Top}{\mathrm{\bf Top}}
\newcommand{\Topp}{\mathrm{\bf Top^\ast}}
\newcommand{\Topw}{\mathrm{\bf Top^w}}
\newcommand{\N}{\mathcal{N}}
\renewcommand{\O}{\mathcal{O}}
\newcommand{\tiret}{\text{--}}  
\newcommand{\WG}{\mathrm{WG}}
\newcommand{\op}{\mathrm{op}}
\newcommand{\bo}{\mathrm{b}}
\newcommand{\secat}{\mathrm{secat}}
\newcommand{\secatWG}{\secat^\WG}
\newcommand{\secatop}{\secat^\op}
\newcommand{\secatb}{\secat^\bo}
\newcommand{\liftcat}{\mathrm{liftcat}}
\newcommand{\liftcatWG}{\liftcat^\WG}
\newcommand{\liftcatop}{\liftcat^\op}
\newcommand{\Dop}{D^\op}
\newcommand{\cat}{\mathrm{cat}}
\newcommand{\catop}{\cat^\op}
\newcommand{\TC}{{\rm TC}\,}
\newcommand{\TCop}{\TC^\op}
\newcommand{\npu}{{n+1}}
\newcommand{\ipu}{{i+1}}
\newcommand{\kpu}{{k+1}}
\newcommand{\id}{{\rm id}}
\newcommand{\inc}{{\rm inc}}
\newcommand{\prun}{{\rm pr}_1}
\newcommand{\prdeux}{{\rm pr}_2}
\newcommand{\inun}{{\rm in}_1}
\newcommand{\indeux}{{\rm in}_2}
\renewcommand{\leq}{\leqslant}
\renewcommand{\geq}{\geqslant}
\newcommand{\inv}{{-1}}

\newcommand{\cercle}{\bm{S}^1}

\newenvironment{myabstract}[1]{%
  \begin{center}
  \begin{minipage}{0.87\textwidth}
  \small
  \noindent\textsc{#1.}%
}{%
  \end{minipage}
  \end{center}
  \par\smallskip
}
\title[Axiomatic sectional category]%
{Axiomatic Sectional category, Topological Complexity, and Homotopic distance}

\author{Jean-Paul Doeraene and Mohammed El Haouari}



\begin{document}

\maketitle

\begin{myabstract}{Abstract}
Many of the properties of sectional category, topological complexity and homotopic distance are in fact derived from a small number of basic properties, which, once established, lead to all the others without further recourse to topology.
On the other hand, there are several variants of these notions: with open covers or else Whitehead-Ganea constructions, also with spaces and maps that are unpointed or else pointed, fibrewise, equivariant, etc.\ or even with algebraic models of spaces and maps.
These are two reasons why we build an axiomatic approach to all these notions, based on just three simple axioms.
\par\hspace{1em}We also introduce the notion of `lifting category' which unifies the notions of sectional category, topological complexity, and homotopic distance, all of which are special cases of lifting category.
\end{myabstract}

\selectlanguage{french}
\begin{myabstract}{Résumé}
Beaucoup des propriétés de la catégorie sectionnelle, de la complexité topologique et de la distance homotopique découlent en réalité d’un petit nombre de propriétés de base qui, une fois établies, permettent de déduire toutes les autres sans recours supplémentaire à la topologie.
D’autre part, il existe plusieurs variantes de ces notions~: avec des recouvrements ouverts ou bien des constructions de Whitehead-Ganea, avec des espaces et applications qui sont non pointés, pointés, fibrés, équivariants, etc., voire avec des modèles algébriques d’espaces et applications.  
Ce sont deux raisons pour lesquelles nous développons une approche axiomatique de toutes ces notions, fondée sur seulement trois axiomes simples.
\par\hspace{1em}Nous introduisons également la notion de \og catégorie de relèvement \fg{} qui unifie les notions de catégorie sectionnelle, de complexité topologique et de distance homotopique, toutes étant des cas particuliers de la catégorie de relèvement.
\end{myabstract}
 
\selectlanguage{english}
The {\em lifting category} from a map $f\colon X \to Y$ to another $\iota\colon A \to Y$ with same codomain is the least integer $n$ such that $X$ can be covered with $n+1$ open subsets $U_i$, $0\leq i \leq n$, and for each $i$, there exists a map $l_i\colon U_i \to A$ such that $\iota \circ l_i$ is homotopic to $f|_{U_i}$. If $f$ and $\iota$ are embeddings, the existence of $l_i$ means that $f(U_i)$ is deformable in $Y$ into $\iota(A)$. On the other hand, if $\iota$ is a fibration, ``homotopic'' can be replaced by ``equal''.

The lifting category is a generalization of both the {\em (Clapp-Puppe) category} of a map and the {\em sectional category} (or {\em Schwarz genus}) of a map.
We define and study the  lifting category with all its particular cases: Lusternik-Schnirelmann category,  sectional category, and also {\em topological complexity} and {\em homotopic distance}, all at once.

\smallskip 
Our aim is to make a clear distinction between topological considerations and, say, category-theoretic considerations. 
We develop this second aspect on the basis of a small number of axioms from which all the expected properties follow. This is done in Sections \ref{liftingcategory} and \ref{homotopicdistance}, and continued in Sections \ref{triangleinequality} and \ref{lscategory}.
The first aspect is dealt with in Sections \ref{liftcatop} to \ref{liftcatwg}.

After unpacking our tools in \secref{tools}, we set out our axioms and discuss the lifting and sectional category in \secref{liftingcategory}.
In \secref{homotopicdistance}, we apply the results of \secref{liftingcategory} to the homotopic distance and the topological complexity.
We then show that the notions based on open covers, see Sections \ref{liftcatop} and \ref{homdistop}, and the notions based on Whitehead and Ganea constructions, see \secref{liftcatwg}, satisfy the axioms. We show that the two approaches are equivalent under a normality condition.
We look at the pointed case in \secref{thepointedcase}. 
In \secref{triangleinequality}, we review the consequences of the triangle inequality of the homotopic distance. 
Finally, we return to the primitive definition of LS~category in \secref{lscategory}. 

\smallskip
For applications and examples, we will limit ourselves to the category $\Top$ of topological spaces and continuous maps, and the category $\Topw$ of well-pointed topological spaces and continuous maps preserving basepoints (`well-pointed' means that the inclusion of the basepoint is a closed cofibration), see \cite{Str72}. 
But the axiomatic development made here can also be applied to other categories provided they fit sufficiently into a `(Quillen) model category' structure.   
We use $\Jc$ as a generic notation for any of these categories.

\section{Useful lemmas}\label{tools}

All constructions here are made up to `homotopy equivalences', in such a way that the numbers we define are invariant with respect to these equivalences. We use the symbol $\simeq$ to indicate that two maps are homotopic. Spaces are non-empty.

\medskip
We will use a lot of `homotopy pullbacks'. Roughly speaking, a homotopy pullback is a homotopy commutative diagram  that satisfies the universal property of pullbacks up to homotopy equivalences.
A homotopy commutative square whose two opposite sides are homotopy equivalences is a homotopy pullback.
Warning: Not every  pullback is a homotopy pullback. However a pullback of two maps, one of which is a fibration, is a homotopy pullback. To build a homotopy pullback of two maps which are not fibrations, we have to replace one of them with a homotopy equivalent fibration. The `homotopy pushout' is the (Eckmann-Hilton) dual notion of the homotopy pullback.
For a more detailed description of this, see \cite{Mat76} in a strict topological context, or \cite{Doe93} in the more general context of `J-categories'.

\begin{lemma}\label{liftlemma}Assume we have any homotopy commutative diagram:
$$\xymatrix@C=3pc{
P\ar[d]_q\ar[r]_p&A\ar[d]^\iota\\
X\ar[r]_f&Y
}$$
and any map $g\colon W \to X$ in $\Jc$.
If there is a map $s\colon W \to P$ such that $q \circ s \simeq g$  then there is a map $l\colon W\to A$ such that $\iota \circ l\simeq f\circ g$.

If in addition the square is a homotopy pullback, then the converse is true too.
\end{lemma}

\begin{proof}
Assume  there exists a map $s \colon W \to P$ such that $q\circ s \simeq g$. Let $l = p \circ s$. We have $\iota \circ l = \iota \circ p \circ s \simeq f \circ q \circ s \simeq f \circ g$.

Conversely, if  there exists $l\colon W\to A$ such that  $\iota \circ l\simeq f\circ g$, then the outer square in the following diagram is homotopy commutative:
$$\xymatrix{
W\ar@{-->}[r]_s\ar[rd]_g\ar@/^1pc/[rr]^{l}&P\ar[r]_p\ar[d]^q&A\ar[d]^\iota\\
&X\ar[r]_f&Y
}$$
If the square is a homotopy pullback, we have a `universal map' $s\colon W \to P$ (defined up to homotopy equivalences) such that $q \circ s \simeq g$ and $p \circ s \simeq l$.
\end{proof}

The map $s$ above, induced by a homotopy pullback, will be called a `whisker map' and denoted by $(g,l)$.

\begin{proposition}[Prism lemma]\cite[Lemmas 12 and 14]{Mat76}\label{prism}
Assume we have any homotopy commutative diagram in $\Jc$:
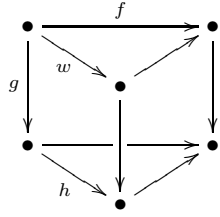
$$\xymatrix@R=1pc{
\bullet\ar[rr]^f\ar[dd]_g\ar[rd]_w&&\bullet\ar[dd]\\
&\bullet\ar[dd]\ar[ru]\\
\bullet\ar[rd]_h\ar[rr]|\hole&&\bullet\\
&\bullet\ar[ru]
}$$
If the right square is a homotopy pullback, then the left square is a homotopy pullback if and only if the rear rectangle is a homotopy pullback.
\end{proposition}

Note that for the diagram to be homotopy commutative, the map $w$ must be the whisker map induced by the right homotopy pullback: $w = (h\circ g, f)$. 

\begin{remark}\label{hpbixid}For any map $\iota\colon A \to Y$ and space $Z$ in $\Jc$, the following pullback is a homotopy pullback:
$$\xymatrix@C=3pc{
A \times Z\ar[r]^{\iota \times \id_Z}\ar@{->>}[d]_\prun&Y \times Z\ar@{->>}[d]^\prun\\
A\ar[r]_{\iota}&Y
}$$
This is because projections are fibrations.

If $\iota$ is a closed cofibration, then $\iota \times \id_Z$ is also a closed cofibration, see \cite[Theorem 12]{Str68} for a more general result.
\end{remark}

\begin{lemma}\label{foursquares}Let $\iota\colon A\to Y$ and $\kappa\colon B \to Z$ be maps in $\Jc$. Then the following pullback is a homotopy pullback:
$$\xymatrix@C=3pc{
A \times B\ar[r]^{\id_A\times \kappa}\ar[d]_{\iota \times \id_B}&A \times Z\ar[d]^{\iota \times \id_Z}\\
Y \times B\ar[r]_{\id_Y\times \kappa}&Y\times Z.
}$$
\end{lemma}

\begin{proof}Consider the following commutative diagram:
$$\xymatrix@C=3pc{
A \times B\ar[r]_{\id_A \times \kappa}\ar@{->>}@/^1pc/[rr]^\prun\ar[d]_{\iota \times \id_B}&A \times Z\ar[d]^{\iota \times \id_Z}\ar@{->>}[r]_\prun &A\ar[d]^\iota\\
Y \times B\ar[r]^{\id_Y\times \kappa}\ar@{->>}@/_1pc/[rr]_\prun&Y\times Z \ar@{->>}[r]^\prun&Y
}$$
By \remref{hpbixid}, the right square and the outer rectangle are homotopy pullbacks. So the left square is a homotopy pullback, too, by the Prism lemma.
\end{proof}

\begin{proposition}\label{hpbkappaxlambda}
Assume we have any homotopy pullback:
$$\xymatrix{
P\ar[r]^p\ar[d]_q&A\ar[d]^\iota\\
X\ar[r]_f&Y
}$$
and any map $\kappa\colon B \to Z$ in $\Jc$. Then the following square is a homotopy pullback, too:
$$\xymatrix@C=3pc{
P\times B\ar[r]^{p\times \kappa}\ar[d]_{q \times \id_B}&A\times Z\ar[d]^{\iota\times \id_Z}\\
X\times B\ar[r]_{f\times \kappa}&Y\times Z
}$$

%
%
%
\end{proposition}

\begin{proof}
Consider the following homotopy commutative diagram:
$$\xymatrix@C=3pc{
P\times B\ar[r]^{p\times \id_B}\ar[d]_{q \times \id_B}&A\times B\ar[r]^{\id_A\times \kappa}\ar[d]_{\iota \times \id_B}&A\times Z\ar[d]^{\iota\times \id_Z}\\
X\times B\ar[r]_{f\times \id_B}&Y\times B\ar[r]_{\id_Y\times \kappa}&Y\times Z
}$$
We can easily see that the square on the left is a homotopy pullback, while \lemref{foursquares} assures us that the square on the right is also a homotopy pullback. So the rectangle is a homotopy pullback too, by the Prism lemma.
\end{proof}

\section[Sectional and Lifting category]{Abstract Sectional and Lifting category}\label{liftingcategory}

In this section, we give the axioms of the (abstract) `sectional category' and the properties that follow directly from them. We also introduce the notion of `lifting category', which generalizes the sectional category.

\begin{axioms}[Sectional category]\label{axioms}
We work with an (abstract) extended natural number $\secat (\iota)$, i.e., a nonnegative integer or $\infty$, defined for any  map  $\iota\colon A \to Y$ in $\Jc$,  called the `sectional category' of $\iota$. 
We assume that  the following three axioms are satisfied:
\begin{itemize}\setlength{\itemsep}{2pt}
\item[\textbf{S0.}] For any map $\iota\colon A \to Y$ in $\Jc$, $\secat (\iota) = 0$ if and only if $\iota$ has a homotopy section, i.e., there is a map $s\colon Y\to A$ such that $\iota\circ s\simeq \id_Y$.
\item[\textbf{S1.}] 
If we have a homotopy commutative diagram in $\Jc$:
$$\xymatrix@R=1pc{
   B \ar[rr]^\zeta\ar[rd]_\kappa && A\ar[ld]^\iota \\
   & X
}$$
then $\secat\, (\iota) \leq \secat\,  (\kappa)$.
\item[\textbf{S2.}] 
If we have a homotopy pullback in $\Jc$:
$$\xymatrix@C=3pc{
P\ar[d]_q\ar[r]_p&A\ar[d]^\iota\\
X\ar[r]^f&Y
}$$
then $\secat(q) \leq \secat(\iota)$.
\end{itemize}
\end{axioms}

Let's already give an elementary but instructive example:\negskip
\begin{definition}\label{boolsecat}
The `boolean sectional category' $\secatb(\iota)$ of a map $\iota\colon A \to Y$ in $\Jc$ is 0 if $\iota$ has a homotopy section, and 1 otherwise.
\end{definition}

It is obvious that $\secatb$ satisfies axioms S0 and S1. If we have a homotopy pullback as in Axiom~S2, and $\iota$ has a homotopy section $s\colon Y \to A$, then $q$ has a homotopy section which is the whisker map $(\id_X,s \circ f)$. So $\secatb$ also satisfies S2.

\begin{definition}[Lifting category] \label{liftcatisasecat} For any pair of maps $f \colon X \to Y$ and $\iota\colon A \to Y$, 
consider a homotopy pullback as in Axiom~S2.
The `lifting category' from $f$ to $\iota$, is the sectional category of the map $q\colon P\to X$. We denote:\\[6pt]
\centerline{$\liftcat_f\,(\iota) = \secat(q)$.}
\end{definition}

\begin{remark}The homotopy pullback $P$ and the map $q$ are only defined up to homotopy equivalences, but if we have another choice: $P'$ and $q'$, we have whisker maps $s\colon P \to P'$ such that $q' \circ s \simeq q$ and  $s'\colon P' \to P$ such that $q \circ s' \simeq q'$, so by Axiom~S1,  $\secat(q') = \secat(q)$.\end{remark}

Although $\liftcat$ is defined above in terms of $\secat$, $\secat$ can also be seen as a special case of $\liftcat$.
Indeed, by applying \defref{liftcatisasecat} with $f = \id_Y$, $p = \id_A$ and $q =\iota$, we have:\negskip
\begin{observation}\label{secatisaliftcat}
For any map $\iota\colon A\to Y$ in $\Jc$,\\[6pt]
\centerline{$\secat(\iota) = \liftcat_{\id_Y}\,(\iota)$.}
\end{observation}

Now let's establish the consequences of axioms S0 to S2.

\begin{proposition} \label{L0} For any pair of maps $f\colon X \to Y$ and $\iota\colon A \to Y$ in $\Jc$, \\
$\liftcat_f (\iota) = 0$ if and only if there is a map $l\colon X\to A$ such that $\iota\circ l\simeq f$.
\end{proposition}

\begin{proof}Use Axiom~S0 and \lemref{liftlemma} with the homotopy pullback of $f$ and $\iota$, and $g = \id_X$.\end{proof}

We can rephrase Axiom~S2 as follows:\negskip
\begin{observation} \label{L1}  For any pair of maps $f\colon X \to Y$ and $\iota\colon A \to Y$ in $\Jc$,\\[6pt]
\centerline{$\liftcat_f\,(\iota) \leq \secat(\iota)$.}
\end{observation}

\begin{proposition}\label{L2}
Assume we have any homotopy commutative diagram in $\Jc$:
$$\xymatrix@C=3pc{
&B\ar[d]_\kappa\ar[r]_\zeta&A\ar[d]^\iota\\
W\ar[r]^g\ar@/_1pc/[rr]_h&X\ar[r]^f&Y.
}$$
Then we have:
$$\liftcat_h\,(\iota) \leq \liftcat_{g}\,(\kappa).$$
If in addition the square is a homotopy pullback (in particular if $f$ and $\zeta$ are homotopy equivalences), then $\liftcat_h\,(\iota) = \liftcat_g\,(\kappa)$.
\end{proposition}

\begin{proof}Using the Prism lemma, extend the homotopy commutative diagram of the hypothesis to this one:
$$\xymatrix@C=3pc{
P''\ar@/_1.5pc/[dd]_{q''}\ar@{-->}[d]\ar[r]&B\ar@/_1.5pc/[dd]_(0.3)\kappa\ar@{-->}[d]\ar[rd]^\zeta\\
P'\ar[d]^{q'}\ar[r]|(0.65)\hole&P\ar[d]^q\ar[r]&A\ar[d]^\iota\\
W\ar[r]^g\ar@/_1pc/[rr]_h&X\ar[r]^f&Y
}$$
where all squares and rectangles are homotopy pullbacks, and dashed arrows are whisker maps. By definition, $\liftcat_h(\iota) = \secat(q')$ and $\liftcat_g(\kappa) = \secat (q'')$. But by Axiom~S1, $\secat(q') \leq \secat(q'')$. So $\liftcat_h(\iota) \leq \liftcat_g(\kappa)$.

If in addition, the square of the hypothesis is a homotopy pullback, then we can assume that $P = B$, $P'' = P'$ and $q'' = q'$. So $\liftcat_h(\iota) = \liftcat_g(\kappa)$.
\end{proof}

\begin{corollary}[Homotopy invariance]\label{liftcathominvariant} Let $g, h\colon W \to Y$ and $\iota, \kappa\colon A \to Y$ be maps in $\Jc$ such that $h \simeq g$ and $\iota \simeq \kappa$. Then $\liftcat_h(\iota)= \liftcat_g(\kappa)$.
\end{corollary}

\begin{proof}This is a special case of \propref{L2} with $\zeta = \id_A$ and $f = \id_Y$.\end{proof}
\medskip
We bring together the essential equalities and inequalities in the following result:
\begin{proposition}\label{inegessent}
Assume we have any homotopy commutative diagram in $\Jc$:
$$\xymatrix@C=3pc{
 B\ar[d]_\kappa\ar[r]_\zeta&A\ar[d]^\iota\\
 X\ar[r]_f&Y
}$$

{\bf (1)} In any case, $\liftcat_f (\iota)\leq \inf \{\secat(\kappa), \secat(\iota)\}$.

{\bf (2)} If the square is a homotopy pullback, then $\liftcat_f (\iota) = \secat(\kappa) \leq\secat(\iota)$.

{\bf (3)} If $f$ has a homotopy section, then $\liftcat_f(\iota) = \secat(\iota) \leq \secat(\kappa)$.

{\bf (4)} If $f$ and $\zeta$ are homotopy equivalences, then $\liftcat_f (\iota)= \secat(\kappa)  = \secat(\iota)$.
\end{proposition}

\begin{proof}
{\bf (1 and 2)} The inequality or equality with $\secat\,\kappa$ is a special case of \propref{L2} with $g=\id_X$. The inequality with $\secat\,\iota$ is \obsref{L1}.

{\bf (3)} If we have a homotopy section $s\colon Y \to X$ of $f$, then (using the Prism lemma) build the following homotopy pullbacks:
$$\xymatrix@C=3pc{
A \ar[d]_\iota\ar[r] &P\ar[d]_q\ar[r]&A\ar[d]^\iota\\
Y \ar[r]_s& X\ar[r]_f&Y
}$$
We have $\secat \,\iota \leq \secat\, q \leq \secat \,\iota$ by Axiom~S2 and $\liftcat_f(\iota) = \secat\,q$ by definition, so $\liftcat_f(\iota) = \secat\, \iota$. The  inequality with $\secat\,\kappa$ is given by (1).

{\bf (4)} If $f$ and $\zeta$ are homotopy equivalences, then both the conditions of (2) and (3) are satisfied, so we have equalities.
\end{proof}

\begin{corollary}\label{secatixid}For any map $\iota\colon A \to Y$ and space $Z$ in $\Jc$, we have:
$$\secat(\iota \times \id_Z) = \secat(\iota).$$ 
\end{corollary}

\begin{proof}The projection $\prun\colon Y\times Z \to Y$ has an obvious section. We get the result applying \propref{inegessent}\petitespace(2 and 3) to the (homotopy) pullback of \remref{hpbixid}.
\end{proof}

\begin{proposition}\label{composition}
For any maps $f\colon X\to Y$, $g\colon W \to X$, and $\iota\colon A \to Y$ in $\Jc$, 
$$\liftcat_{f\circ g}\,(\iota) \leq \liftcat_{f}\,(\iota) \leq \secat(\iota).$$
If $g$ is a homotopy equivalence, then $\liftcat_{f\circ g}\,(\iota) = \liftcat_f\,(\iota)$. 
\end{proposition}

\begin{proof}Build the following homotopy pullbacks (using the Prism lemma):
$$\xymatrix@C=3pc{
P' \ar[d]_{q'}\ar[r]_{p'} &P\ar[d]_q\ar[r]_p&A\ar[d]^\iota\\
W \ar[r]_g& X\ar[r]_f&Y
}$$
By definition of lifting category, we have $\liftcat_f\,(\iota) = \secat(q)$ and $\liftcat_{f\circ g}\,(\iota) =\secat(q')$, and by Axiom~S2, we also have $\secat (q') \leq \secat(q) \leq \secat(\iota)$.

Now if $g$ is a homotopy equivalence, then so is $p'$ and by \propref{inegessent}\petitespace(4), $\secat(q') = \secat(q)$,  so $\liftcat_{f\circ g}\,(\iota) = \liftcat_{f}\,(\iota)$.
\end{proof}

\begin{proposition}\label{liftcatproduit}
For any maps $f\colon X\to Y$, $\iota\colon A \to Y$ and $\kappa\colon B \to Z$ in $\Jc$, 
$$\liftcat_{f \times \kappa} \,(\iota \times \id_Z) = \liftcat_f \,(\iota).$$
\end{proposition}

\begin{proof}
Build the homotopy pullback of $f$ and $\iota$. Looking at the two homotopy pullbacks of \propref{hpbkappaxlambda}, by definition of lifting category, we get $\liftcat_f\,(\iota) = \secat (q)$ and $\liftcat_{f \times \kappa}\, (\iota \times \id_Z) = \secat(q \times \id_B)$. And by  \cororef{secatixid}, we know that $\secat(q \times \id_B) = \secat(q)$.
\end{proof}

\section{Homotopic distance and Topological complexity}\label{homotopicdistance}
Once we have our lifting category and its properties established in the previous section, we can apply them to the lifting category to a diagonal map, which is an outstanding special case. 

\smallskip
Note that any map $w\colon X \to Y\times Y$ is in fact a whisker map $(f,g)$ of two maps $f, g\colon X \to Y$ in the following (homotopy) commutative diagram:
$$\xymatrix{
X\ar@{-->}[r]_w\ar[rd]_f\ar@/^1pc/[rr]^{g}&Y\times Y\ar[r]_(.55)\prdeux\ar[d]^\prun&Y\ar[d]\\
&Y\ar[r]&\ast
}$$
(We use $\ast$ to denote the one-point space, final object of $\Jc$.)

\begin{definition}[Homotopic distance] \label{defhomdist} Let $f, g\colon X\to Y$ be two maps in $\Jc$. The `homotopic distance' between $f$ and $g$ is the lifting category from $(f,g)$ to $\Delta$, where $(f,g)\colon X \to Y \times Y$ is the whisker map induced by the product $Y \times Y$ and $\Delta\colon Y \to Y\times Y$ is the diagonal map. We denote  $$D(f,g) = \liftcat_{(f,g)}\, (\Delta).$$
\end{definition}

\begin{proposition}\label{Dhominvariant}For any maps $f, f', g, g'\colon X \to Y$ in $\Jc$ such that $f\simeq f'$ and $g\simeq g'$, we have $D(f,g) = D(f',g')$.
\end{proposition}

\begin{proof}As any map into the product $Y\times Y$ is determined by the projections to the factors, the homotopies from $f$ to $f'$ and from $g$ to $g'$ induce a homotopy from $(f,g)$ to $(f',g')$. The result follows from \cororef{liftcathominvariant}. 
\end{proof}

\begin{definition}[Topological complexity] \label{defTC} Let $Y$ be a space in $\Jc$. The `topological complexity' of $Y$ is $\secat\,(\Delta)$ where $\Delta \colon Y\to Y\times Y$ is the diagonal map. We denote  $$\TC(Y) = \secat(\Delta).$$
\end{definition}

\begin{remark}\label{TCestunedistance}Actually, $\TC(Y) = D(\prun,\prdeux)$ where $\prun, \prdeux\colon Y\times Y \to Y$ are the two projections. 
Indeed, the whisker map $(\prun,\!\prdeux) = \id_{Y\!\times\! Y}$, so 
$D(\prun,\prdeux) = \liftcat_{\id_{Y\!\times\!Y}} (\Delta) = \secat (\Delta) = \TC(Y)$.
\end{remark}

\begin{proposition}\label{homdistisasecat}
 For any  maps $f, g\colon X\to Y$  in $\Jc$, we have
 $$D(f,g) = \liftcat_\Delta(q) = \secat(q') \leq \liftcat_{f\times g}(\Delta) = \secat(q) \leq \TC(Y)$$
 where $q$ and $q'$ are the maps in the following homotopy pullbacks:
$$\xymatrix@C=3pc{
P'\ar[r]_{p'}\ar[d]_{q'}&P\ar[r]_{p}\ar[d]_q&Y\ar[d]^\Delta\\
X\ar@/_1pc/[rr]_{(f,g)}\ar[r]^\Delta&X\times X\ar[r]^{f\times g}&Y\times Y.
}$$
\end{proposition}

\begin{proof}Apply definitions and Axiom~S2.\end{proof}

\begin{proposition}\label{separation}$D(f,g) = 0$ if and only if $f \simeq g$.
\end{proposition}

\begin{proof} First note that by \propref{L0}, $D(f,g) = 0$ if and only there is a map $l\colon X \to Y$ such that $\Delta \circ l \simeq (f,g)$.

So if $D(f,g) = 0$, composing with the first projection we get $l = \prun \circ \Delta \circ l \simeq \prun \circ (f,g) = f$ and similarly we get $l \simeq g$, so $f\simeq g$.

Conversely, if $f\simeq g$, by \propref{Dhominvariant}, $D(f,g) = D(f,f)$, which is 0 since $\Delta \circ f = (f,f)$. 
\end{proof}

\begin{proposition}\label{symetrie} $D(f,g) = D(g,f)$.
\end{proposition}
\begin{proof}Consider the following homotopy commutative diagram:
$$\xymatrix@C=3pc{&Y\ar[d]_(.4)\Delta\ar@{=}[r]&Y\ar[d]^(.4)\Delta\\
X\ar[r]^(.4){(f,g)}\ar@/_1pc/[rr]_{(g,f)}&Y\times Y\ar[r]^{(\prdeux,\prun)}&Y \times Y
}$$
Note that $(\prdeux,\prun)$ switches the copies of $Y$. It is a homeomorphism since it is its own inverse, so the result follows from \propref{L2}.
\end{proof}

\begin{proposition}\label{compositionadroite}
For any maps $f, g\colon X \to Y$ and $h\colon W \to X$ in $\Jc$, we have
$$D(f\circ h, g \circ h) \leq D(f,g)$$
and if $h$ is a homotopy equivalence, we have the equality.
\end{proposition}

\begin{proof}Note that $(f\circ h, g\circ h) \simeq (f,g) \circ h$ and apply \propref{composition}. 
\end{proof}

\begin{proposition}\label{compositionagauche}
For any maps $f, g\colon W \to X$ and $h\colon X \to Y$ in $\Jc$, we have:
$$D(h\circ f, h \circ g) \leq D(f,g)$$
and if $h$ is a homotopy equivalence, we have the equality.
\end{proposition}

\begin{proof}Apply \propref{L2} with the commutative diagram:
$$\xymatrix@C=3pc{
&X\ar[d]_\Delta\ar[r]_h&Y\ar[d]^\Delta\\
W\ar[r]^{(f,g)}\ar@/_1pc/[rr]_{(h \circ f, h \circ g)}&X\times X\ar[r]^{h\times h}&Y \times Y
}$$
\end{proof}

\begin{proposition}
For any maps $f, g\colon X \to Y$ and $f', g'\colon X' \to Y'$ in $\Jc$, we have: 
$$D(f\times f', g\times g') = D(f\times g',g\times f').$$\end{proposition}

\begin{proof}Consider the following communative diagram:
$$\xymatrix@C=3pc{&Y\times Y'\ar[d]_(.4)\Delta\ar@{=}[r]&Y\times Y'\ar[d]^(.4)\Delta\\
X\times X'\ar[r]^(.4){(f\!\times\!f'\!,g\!\times\!g')}\ar@/_1pc/[rr]_{(f\!\times\!g'\!,g\!\times\!f')}&(Y\!\times\!Y') \times (Y\!\times\!Y')\ar[r]^\sigma&(Y\!\times\!Y') \times (Y\!\times\!Y')
}$$
where $\sigma$ is the homeomorphism which leaves the copies of $Y$ in place and switches those of $Y'$.
The result is given by \propref{L2}.
\end{proof}

\begin{proposition}\label{timesh}For any maps $f,g\colon X \to Y$ and $h\colon X' \to Y'$ in $\Jc$, we have $$D(f\times h, g \times h) = D(f,g).$$
\end{proposition}

\begin{proof}Consider the following commutative diagram:
$$\xymatrix{
&&Y\times Y'\ar[d]_{\Delta\times\id_{Y'}}\ar@{=}[r]&Y\times Y'\ar[d]^{\Delta}\\
X\times X'\ar[rr]^(.45){(f,g)\times h}\ar@/_1pc/[rrr]_{(f\times h, g\times h)}&&(Y\!\times\! Y)\times Y'\ar[r]^(0.45)\delta&(Y\!\times\! Y')\times (Y\!\times\! Y')
}$$
where $\delta = (\prun\times\id, \prdeux\times\id)$ duplicates $Y'$. Applying \propref{L2} to it, we get $D(f\times h, g \times h) \leq \liftcat_{(f,g)\times h}\,(\Delta \times \id)$. But by \propref{liftcatproduit} the latter is $\liftcat_{(f,g)}\,(\Delta) = D(f,g)$. So we have $D(f\times h, g \times h) \leq D(f,g)$.

On the other hand, $f = \prun  \circ (f\times h) \circ \inun$ and $g = \prun  \circ (g\times h) \circ \inun$ where $\inun \colon X \to X \times X'$ is a section of the projection $X \times X' \to X$, and $\prun \colon Y \times Y' \to Y$ is the first projection. By \propref{compositionagauche} and \propref{compositionadroite}, we get $D(f,g) \leq D(f\times h, g \times h)$.
\end{proof}

\begin{proposition}\label{etonnant}For any maps $f, g \colon X \to Y$ in $\Jc$, 
$$D(f,g) = \liftcat_{(\id_X,f)} (\id_X,g).$$
\end{proposition}

\begin{proof}Using the Prism lemma in the following commutative diagram:
$$\xymatrix@C=3pc{
  & X \ar[r]^g\ar[d]_(.4){(\id_X,g)} & Y\ar[d]_\Delta\ar@/^2pc/[dd]^{\id_X} \\
X\ar[r]^(0.4){(\id_X,f)}\ar@/_1pc/[rr]_(0.3){(g,f)} &X\times Y\ar[r]^{g \times \id_Y}\ar@{->>}[d]|(0.3)\hole^(0.65)\prun & Y\times Y\ar@{->>}[d]_\prun \\
  &X \ar[r]^g & Y
}$$
we see that the upper square is a homotopy pullback. The result follows from \propref{L2}.
\end{proof}

\begin{corollary}\label{minorationD}For any maps $f, g \colon X \to Y$, we have:
$$D(f,g) \geq \sup\{\liftcat_f (g), \liftcat_g (f) \}.$$
\end{corollary}

\begin{proof}Apply \propref{L2} to the following commutative diagram:
$$\xymatrix@C=3pc{
 & X \ar@{=}[r]\ar[d]_(.4){(\id_X,g)} &X\ar[d]^g\\
X\ar[r]^(0.4){(\id_X,f)}\ar@/_1pc/[rr]_f &X\times Y\ar[r]^\prdeux&Y.
}$$
\end{proof}

\begin{remark}The inequality can be strict, even when $\liftcat_f(g) = \liftcat_g(f)$. For instance consider the two projections $\prun, \prdeux\colon Y\times Y \to Y$. We have $\liftcat_{\prun} (\prdeux) = \liftcat_{\prdeux} (\prun) = 0$ while $D(\prun,\prdeux) = \TC(Y)$.\end{remark}

\section[Open covers]{Lifting category with open covers}\label{liftcatop}

In this section, we give a version of the lifting category based on open covers, which is an extension of the (more or less) original definition of the sectional category. 
We show that axioms S0 to S2 are well satisfied. 

\begin{definition}
For continuous maps $f\colon X \to Y$ and $\iota\colon A \to Y$, let us denote  $\liftcatop_f \, (\iota)$ the least integer $n$ for which there exists an open cover $(U_i)_{0\leq i \leq n}$ of $X$ and maps $l_i\colon U_i \to A$ such that $\iota \circ l_i \simeq f|_{U_i}$ for all $i$. Such a cover is said to be `categorical'. If no categorical cover exists for any $n$, we denote $\liftcatop_f\,(\iota) = \infty$.

If $f = \id_Y$, we write $\secatop \, (\iota)= \liftcatop_{\id_Y}\, (\iota)$.
\end{definition}

\begin{remark}\label{L0op} As there is only one cover of $X$ with one open set (which must be $X$ itself), $\liftcatop_f(\iota) = 0$ if and only if there is map $l\colon X \to A$ such that $\iota \circ l\simeq f$. 
\end{remark}


It is obvious that:\negskip
\begin{observation}\label{liftcatophominvariant} Let $f, f'\colon X \to Y$ and $\iota, \iota'\colon A \to Y$ be maps such that $f \simeq f'$ and $\iota \simeq \iota'$. Then $\liftcatop_f(\iota)= \liftcatop_{f'}(\iota')$.
\end{observation}

It should be noticed that if we work in $\Topw$, all maps preserve basepoints, so each $U_i$ is supposed to include the basepoint of $X$; also homotopies are defined on the reduced cylinder, and preserve basepoints like all maps, in other words these are pointed homotopies.
However, if we look carefully at the proofs, we see that being in $\Top$ or in $\Topw$ makes no difference.
We'll come back to this issue in \secref{thepointedcase}.

\begin{observation}\label{deformation} If $f\colon X \to Y$ and $\iota\colon A \to Y$ are embeddings, then an open cover  $(U_i)_{0\leq i \leq n}$ of $X$ is categorical for $\liftcatop_f(\iota)$ if and only if for all $i$, $f(U_i)$ is deformable in $Y$ into $\iota(A)$. 
\end{observation}

Indeed if $f$ and $\iota$ are embeddings, for any open subset $U$ of $X$, we can identify $f(U)$ with $U$ and $\iota(A)$ with $A$.
 A deformation $H \colon f(U) \times I \cong U \times I \to Y$ such that $\forall~u \in U$,  $H(u, 0) = u$ and $H(u,1) \in \iota(A) \cong A$, corresponds to a homotopy between $H(\tiret, 0) = f|_{U}$ and $g = H(\tiret, 1)$ such that $g = \iota \circ l$ with $l\colon U \to A$,  $l(u) = \iota^\inv (g(u))$.

\begin{example}Let $\iota \colon  \cercle_d \to \cercle$ be the identity between the circle with the discrete topology and the circle with the usual topology. Obviously, as $\iota(\cercle_d) = \cercle$,  $\cercle$ is deformable in itself into $\iota(\cercle_d)$ with the static homotopy $H$. However, $H(\tiret,1) = \id_{\cercle}$ does not factors through $\cercle_d$; this is because $\iota$ is not an embedding. In fact, only constant maps  $\cercle \to \cercle_d$ are continuous, so only covers of $\cercle$ with open subsets contractible in $\cercle$ are categorical.  We need at least two such open subsets to cover $\cercle$.
So $\secatop(\iota) = 1$.  
\end{example}

Since we've defined $\secatop$ in terms of $\liftcatop$, we need to check that it's consistent with \defref{liftcatisasecat}:

\begin{lemma}\label{liftcatopisasecat}
If we have a homotopy pullback in $\Jc$:
$$\xymatrix@C=3pc{
P\ar[d]_q\ar[r]_p&A\ar[d]^\iota\\
X\ar[r]^f&Y
}$$
then $\liftcatop_f(\iota) = \secatop(q)$.
\end{lemma}

\begin{proof}Apply \lemref{liftlemma} with $g = \inc_i \colon U_i \hookrightarrow X$, the inclusion of any $U_i$ into $X$.
There is a map $l_i\colon U_i \to A$ such that $\iota \circ l_i \simeq f \circ \inc_i = f|_{U_i}$ if and only if
there is a map $s_i \colon U_i \to P$ such that $q \circ s_i \simeq \inc_i = \id_X|_{U_i}$. So a cover is categorical for $\liftcatop_f(\iota)$ if and only if it is categorical for $\secatop(q)$.
\end{proof}

\begin{proposition}
The integer $\secatop$  satisfies Axioms S0 to S2.
\end{proposition}

\begin{proof}
{\bf S0.} This is \remref{L0op} with $f = \id_X$.

{\bf S1.} Let $\kappa \simeq \iota \circ \zeta$, as in the statement of Axiom~S1. Assume we have an open cover $(U_i)_{0\leq i \leq n}$ of $X$ and maps $s_i \colon U_i \to B$ such that $\kappa \circ s_i \simeq \inc_i \colon U_i \hookrightarrow X$. This means that $\iota \circ (\zeta \circ s_i) \simeq \inc_i$. So if an open cover is categorical for $\secatop(\kappa)$, it is categorical for $\secatop(\iota)$.

{\bf S2.} Consider a homotopy pullback as in  the statement of Axiom~S2. Assume we have an open cover $(V_i)_{0\leq i \leq n}$ of $Y$ and maps $s_i \colon V_i \to A$ such that $\iota \circ s_i \simeq \id_Y|_{V_i}$. Let $U_i = f^{-1}(V_i)$. Then  $(U_i)_{0\leq i \leq n}$ is an open cover of $X$, and for any $i$, we have $\iota\circ (s_i \circ f|_{U_i}) \simeq f|_{U_i}$. Hence if there is a categorical open cover for $\secatop(\iota)$, then there is a categorical open cover for $\liftcatop_f(\iota)$, which is $\secatop(q)$ by \lemref{liftcatopisasecat}.
\end{proof}

\section{Homotopic distance with open covers}\label{homdistop}

Once we have defined $\liftcatop$ (in the previous section), we automatically obtain a $\Dop$, coming from \defref{defhomdist}, and a $\TCop$, coming from \defref{defTC}:

\begin{definition}\label{defDop}
For any space $Y$ and for any maps $f, g\colon X \to Y$ in $\Jc$,\\
 $\TCop(Y) = \secatop(\Delta)$ and $\Dop(f,g) = \liftcatop_{(f,g)} (\Delta)$  where $\Delta\colon Y \to Y \times Y$ is the diagonal map.
\end{definition}

\begin{remark}From \obsref{deformation},  $\TCop(Y)$ is the least integer $n$ such that $Y\times Y$ can be covered by $n+1$ open subsets, each of them being deformable in $Y\times Y$ into $\Delta(Y)$.
\end{remark}

Originally, the topological complexity was defined by Farber \cite{Far03} and the homotopic distance was defined by Mac{\'i}as-Virg{\'o}s and Mosquera-Lois \cite{MacMos21}.
We want to check that \defref{defDop} is consistent with these definitions.

\medskip
Actually, there is an alternative to the diagonal which is the `(free) path fibration' $\pi$ in the following diagram:
$$\xymatrix{
   Y\hspace{2mm} \ar[rr]^\Delta  \ar[rd]^s && Y\times Y\\
   & Y^I \ar@{->>}[ru]_\pi \ar@/^0.8pc/[lu]^e
}$$
In this diagram, $\Delta = \pi\circ s$, $e \circ s = \id_Y$ and $s \circ e \simeq \id_{Y^I}$.
The map $\pi$ is a `fibration replacement' of $\Delta$.
(The space $Y^I$ is well-pointed if $Y$ is well-pointed, see \cite[Lemma 4]{Str72}.)

\begin{lemma}\label{pathfibration}
With these data, we have $D(f,g) = \liftcat_{(f,g)}\,(\Delta) = \liftcat_{(f,g)} \,(\pi)$ and $\TC(Y) = \secat\,(\Delta) = \secat \,(\pi)$.
\end{lemma}

\begin{proof}
These are special cases of \cororef{liftcathominvariant} and \propref{inegessent}(4).
\end{proof}

These characterizations of homotopic distance and topological complexity as lifting categories to the path fibration are the link with the original definitions.
Indeed, $\secatop(\pi)$ is exactly Farber's definition of topological complexity, and for the homotopic distance, the equality $\Dop(f,g) = \liftcatop_{(f,g)}(\pi)$ of \lemref{pathfibration} allows us to prove the following result that matches Mac{\'i}as-Virg{\'o}s and Mosquera-Lois's definition of homotopic distance:

\begin{proposition}\label{dopeqliftcatop} Let $f, g\colon X \to Y$ be maps in $\Jc$. Then $\Dop(f,g)$ is the least integer $n$ for which there exists an open cover $(U_i)_{0\leq i \leq n}$ of $X$ such that $f|_{U_i} \simeq g|_{U_i}$ for all $i$. If no such cover exists for any $n$, $\Dop(f,g) = \infty$.
\end{proposition}

\begin{proof}If we have a homotopy $H_i \colon U_i \times I \to Y$ between $f|_{U_i}$ and $g|_{U_i}$, then we can define $l_i \colon U_i \to Y^I$ by $l_i(x) = H_i(x, \tiret)$ and we have $\pi \circ l_i = (f,g)|_{U_i}$.

Conversely if we have a map $l_i \colon U_i \to Y^I$ such that $\pi \circ l_i \simeq (f,g)|_{U_i}$, then we also have a map $l'_i \colon U_i \to Y^I$ such that $\pi \circ l'_i = (f,g)|_{U_i}$ (because $\pi$ is a fibration), and we can define a homotopy $H_i$ between $f|_{U_i}$ and $g|_{U_i}$ by $H_i(x,t) = l'_i(x)(t)$.
\end{proof}




\section{The pointed case}\label{thepointedcase}

As noted above, in $\Top$, the open sets of a categorical cover are not supposed to include any specific point, while in $\Topw$ they must include the basepoint; also in $\Top$, homotopies are free while in $\Topw$, homotopies are pointed.

By the way, we don't want to work in the category $\Topp$ of pointed spaces and maps, because there, the usual pointed homotopy does not fit entirely with the `model category' framework. See \cite[Chapter 10]{Str11}.
 
 \medskip
 {\em  A priori}, (apart from the trivial case where $x_0$  is isolated) there is no reason why the lifting category should be equal whether it is calculated in $\Top$ or in $\Topw$; the latter could be greater than the former. To avoid ambiguity, we will write $\liftcat^\circ$  when the lifting category is calculated in $\Top$, and $\liftcat^\ast$ when it is calculated in $\Topw$ (we drop the superscript ``\raisebox{2pt}{\footnotesize op}'').
 Fortunately, we have the following result:
 
 \begin{theorem}\label{pointedcase} Let $f\colon X\to Y$, $X$ normal, and $\iota\colon A \to Y$ be maps in $\Topw$.
 If $\liftcat^\circ_f (\iota) = n > 0$ when $f$ and $\iota$ are regarded in $\Top$ (via the forgetful functor), then $\liftcat^\ast_f(\iota) = n$. 
 \end{theorem}
 
 In this theorem and others below, we require $X$ to be normal, i.e., any two disjoint closed subsets of $X$ have disjoint open neighborhoods; we do {\em not} require $X$ to be Hausdorff. Normality is not a homotopy invariant, it can be annoying. But since the lifting category is a homotopy invariant, we can replace ``normal'' by ``with same (strong) homotopy type as a normal space''.

\begin{proof} 
We can assume that $X$ is connected; if it were not, it would suffice to stick to the connected component of its basepoint $x_0$.

By hypothesis, there exists an open cover $(U_i)_{0\leq i \leq n}$ of $X$, $U_i \subsetneq X$, with maps $l_i\colon U_i \to A$ and homotopies $H_i\colon U_i \times I \to Y$ between  $f|_{U_i}$ and $\iota \circ l_i$ for all $i$. But not every $U_i$ contains the basepoint $x_0$ of $X$, the maps $l_i$ are not pointed, and the homotopies $H_i$ are free. 

As $X$ is well-pointed, $\{x_0\}$  is closed in $X$ and there is an open neighbourhood $\N$ of $x_0$ deformable in $X$ into $\{x_0\}$ with some pointed homotopy $G$ rel $\{x_0\}$. We may assume that $x_0 \not\in \overline{U_i}\backslash U_i$ for any $i$ because if $x_0 \in \overline{U_i}\backslash U_i$ for some $i$, the normality of $X$ allows us to replace $U_i$ with a bit smaller open subset, keeping the covering of $X$. Assume that $x_0 \in U_i$ for $i  \leq k$ ($0 \leq k \leq n$) and  $x_0 \not\in U_i$ for $i > k$. Let:
$$\O = \N \cap U_0 \cap \dots \cap U_k \cap \complement \overline{U_\kpu} \cap \dots \cap \complement \overline{U_n} \subsetneq X.$$
The basepoint $x_0$ belongs to $\O$, and by normality of $X$, we have two open subsets $V$ and $W$ such that:
$$ x_0 \in W \subsetneq \overline{W} \subsetneq V \subsetneq \overline{V} \subsetneq \O.$$
(The connectedness of $X$ ensures that all these inclusions are strict.) Let $S_0 = W$ and $S_i = V$ for each $i > 0$, and let $U_i' = (U_i\backslash \overline{S_i}) \cup S_i$ for all $i$. Note that $U_i\backslash \overline{S_i}$ and $S_i$ are disjoint open sets, the latter containing $x_0$.  Actually, if $i > k$, $\overline{S_i} = \overline{V}\subset \complement \overline{U_i}$, so $U_i \backslash \overline{S_i} = U_i$. The opens $(U_i')_{0\leq i\leq n}$ cover $X$ because $U_i \subset U_i'$ for $i >k$, and for $i \leq k$, the difference between $U_i'$ and $U_i$ is $\overline{S_i} \backslash S_i$, but $\overline{V}\backslash V \subset U_0'$ and  $\overline{W}\backslash W \subset U_1'$, so it's good. We can now define pointed homotopies $K_i \colon U_i' \times I \to X$:
\[ 
K_i(x,t) = \begin{cases} 
H_i(x,t) & \text{if } x \in U_i\backslash \overline{S_i}   \\
(f \circ G) (x,t)& \text{if } x \in S_i.
\end{cases}
\]
We have $K_i(\tiret,0) = f|_{U_i'}$. On the other hand, if $x \in U_i\backslash \overline{S_i}$, $K_i(x,1) = H_i(x,1) = (\iota \circ l_i)(x)$ and if $x \in S_i$,  $K_i(x,1) = f(x_0) = y_0 = \iota(a_0)$, where $y_0$ is the basepoint of $Y$ and $a_0$ is the basepoint of $A$. So $K_i$ is a pointed homotopy between $f|_{U_i'}$ and $\iota \circ l_i'$ where $l_i' \colon U_i' \to A$ is defined by:
\[ 
l_i'(x) = \begin{cases} 
l_i(x) & \text{if } x \in U_i\backslash \overline{S_i}   \\
a_0 & \text{if } x \in S_i.
\end{cases}
\]
\end{proof}

\begin{remark}\label{caszero} If $\liftcat^\circ_f (\iota) = 0$, the argument of \thmref{pointedcase} fails because it requires at least two open sets, but we can nevertheless prove that $\liftcat^\ast_f (\iota) \leq 1$ by setting $U_0 = X$ and $U_1 = \N$. ($\N$ can not be $X$ if we assume $\liftcat^\ast_f (\iota) \not= 0$.) \end{remark}

\begin{example}Consider the space $Y$ that is a curved tube, the ends of which being attached to a wedge of circles $X\vee A$. The circles $A$ and $X$ are then subspaces of $Y$. All three spaces are pointed with the attachment point $y_0$ of the wedge. We can describe $Y$ as the square $I \times I$ with the identifications $(x, 0) \sim (x, 1)$ and $(0,0) \sim (1,0)$. 
Let $\langle x,y\rangle$ denote the equivalence class of $(x,y)$ in $Y$.
Then $y_0 = \langle 0, 0 \rangle$, $X = \{\langle 0,y\rangle \in Y\}$  and $A = \{\langle1,y\rangle \in Y\}$.

Let $\iota \colon A \hookrightarrow Y$ be the inclusion. The space $Y$ cannot be deformed in itself into $A$ by any homotopy.  But the two half tubes $T_0 = \{\langle x,y\rangle \in Y \mid x \leq \frac12\}$ and $T_1 = \{\langle x,y\rangle \in Y \mid x \geq \frac12\}$ cover $Y$, and we can find open subsets $U_i$ of $Y$ with $T_i \subset U_i$ ($i = 0$ or $1$), each $U_i$  being deformable in $Y$ into $T_i$, itself deformable in $Y$ into $A$. So $\secat^\circ(\iota) = 1$ and by \thmref{pointedcase}, also $\secat^\ast(\iota) = 1$.

Now let  $f \colon X \hookrightarrow Y$ be the inclusion and  let $l\colon X \to A \colon \langle 0,y\rangle \mapsto \langle 1,y\rangle$ be the obvious homeomorphism between $X$ and $A$. We have a free homotopy $H\colon X \times I \to Y\colon (\langle 0,y\rangle, t) \mapsto  \langle t,y\rangle$ between $f$ and $\iota \circ l$. So $\liftcat^\circ_f(\iota) = 0$.  But $X$ is not deformable into $A$ in $Y$ by any pointed homotopy, so $\liftcat^\ast_f(\iota) \not= 0$. And by \obsref{L1},  $\liftcat^\ast_f (\iota) = 1$.

Let $g = \iota \circ l$. As $l$ is a homeomorphism, the equalities above give $\liftcat^\circ_f(g) = 0$ and $\liftcat^\ast_f(g) = 1$. Also by \propref{separation}, $D^\circ (f,g) = 0$, as $f\simeq g$ with the free homotopy $H$. 
On the other hand, \cororef{minorationD} and \remref{caszero} show that  $D^\ast(f,g) = 1$.
\end{example}

The situation where the lifting category is zero in $\Top$ and not zero in $\Topw$ cannot occur when $A$ is a singleton $\{a_0\}$. In this case (which we will return to in \secref{lscategory}), there is only one map $l\colon X \to A$ and only one map $\iota \colon A \to Y$ in $\Topw$ (because $\iota(a_0)$ must be the basepoint $y_0$ of $Y$), and we have the following result:
\begin{proposition}Let $f\colon X\to Y$ and $\iota \colon A=\{a_0\} \to Y$ be maps in $\Topw$.
 Then $\liftcat^\circ_f (\iota) =  0$ when $f$ and $\iota$ are regarded in $\Top$ (via the forgetful functor) if and only if $\liftcat^\ast_f (\iota) = 0$. 
\end{proposition}

\begin{proof}As always, $x_0$ (resp. $y_0$) is the basepoint of $X$ (resp. $Y$).
 Let's assume we have a free homotopy $F \colon X \times I \to Y$ between $f$ and $\iota \circ l$ where $l\colon X \to \{a_0\}$ is the constant map; so $\forall~x \in X$, $F(x,0) = f(x)$ and $F(x,1) = (\iota \circ l)(x) = y_0$.

Let $Z = X \times \{0,1\} \cup \{x_0\} \times I$. We can define a homotopy $G\colon Z \times I \to Y$ by: 
$$\begin{cases}
G(x, 0, t) = f(x) & \text{ for } x \in X, t \in I;\\
G(x, 1, t) = F(x_0, 1-t)& \text{ for } x \in X, t \in I;\\
G(x_0, s, t) = F(x_0, (1-t)s)& \text{ for } s \in I, t \in I.\end{cases}$$  
The map $G$ is well defined because if $x = x_0$ and $s = 0$ (resp. $s = 1$), the third line matches the first (resp. second) line, and it is continuous because $X\times \{0\}$, $X\times \{1\}$ and $\{x_0\} \times I$ are closed subsets of $Z$. 

Note that $G|_{Z\times \{0\}} = F|_Z$. As the inclusion $\{x_0\} \hookrightarrow X$ is a cofibration,  the inclusion $Z \hookrightarrow X \times I$ is also a cofibration.
 By the homotopy extension property, $G$ can be extended to a homotopy $G' \colon (X \times I) \times I \to Y$. 
 
We can define a homotopy $F': X \times I \to Y$ by $F'(x,s) =  G'(x, s, 1)$. As $G'|_{Z \times I}=G$, we have:
$$\begin{cases}
F'(x,0) = G(x, 0, 1) = f(x);\\
F'(x,1) = G(x, 1, 1) = F(x_0, 0) = f(x_0) = y_0 = (\iota \circ l)(x);\\
F'(x_0,s) = G(x_0, s, 1) = F(x_0, 0) = f(x_0).\end{cases}$$ 
So $F'$ is a pointed homotopy between  $f$ and the constant map $\iota \circ l$.
\end{proof}


\section{Whitehead and Ganea constructions}\label{liftcatwg}
In this section, we give a version of the lifting category based on the constructions of Whitehead and Ganea. We show that axioms S0 to S2 are well satisfied.
We also examine how this version relates to the version with open covers.

\begin{definition}\label{whitehead}
For any map $\iota \colon A \to Y\in\Jc$, the \emph{Whitehead
construction} associated to $\iota$ is the
following sequence of homotopy commutative diagrams ($i> 0$),
starting with $t_0 = \iota \colon A \to Y$:
$$\xymatrix{
&Y^i\times A
\ar[rd]
\ar@/^5pt/[rrrd]^{{\rm id}_{Y^i} \times \iota}&&&\\
T^{i-1}(\iota) \times A
\ar[ru]^{t_{i-1}\times{\rm id}_A}
\ar[rd]_{{\rm id}_{T^{i-1}(\iota)}\times\iota}
&&T^i(\iota)\ar[rr]|-(.35){t_i}&&Y^{i+1}\\
&T^{i-1}(\iota)\times Y\ar[ru]
\ar@/_5pt/[rrru]_{t_{i-1}\times{\rm id}_Y}&&&
}$$
where the outer square is a homotopy pullback (by \lemref{foursquares}),
the inner square is a homotopy pushout (by construction of $T^i(\iota)$),
and the map $t_i\colon T^i(\iota) \to Y^\ipu = Y\times \dots\times Y$ ($\ipu$ times) is the whisker map induced by this homotopy pushout.
The spaces $T^i(\iota)$ are called the {\em fat wedges} of $\iota$.
\end{definition}

The Whitehead construction is well defined up to homotopy equivalences because homotopy pullbacks and homotopy pushouts are well defined up to homotopy equivalences. 

\begin{remark}\label{Tnstandard}
Actually, if $\iota$ is a closed cofibration, the Whitehead construction may be carried out with true pullbacks and true pushouts, and all involved maps are then closed cofibrations. This is mainly due to \remref{hpbixid} and \cite[Theorem~6]{Str68}.  Then, as cofibrations are embeddings, $T^i(\iota)\cong \{(y_0,\dots,y_i) \in Y^\ipu \mid y_j \in A \text{ for at least one } $j$\}$. See also \cite[Corollary~11]{GarVan10}.
\end{remark}

Here is the Whitehead version of the lifting and sectional category:
\begin{definition}For any maps $f\colon X \to Y$ and $\iota\colon A\to Y$ in $\Jc$, we denote $\liftcatWG_f (\iota)$ the least integer $n$ such that there is a map $l \colon X \to T^n(\iota)$ with $t_n \circ l \simeq \Delta_\npu \circ f$, where $\Delta_\npu$ is the diagonal map  $\Delta_\npu\colon Y \to Y^\npu$. If no such $n$ exists, we denote  $\liftcatWG_f (\iota) = \infty$.

If $f = \id_Y$, we denote $\secatWG \, (\iota) = \liftcatWG_{\id_Y} (\iota)$.
\end{definition}

\begin{remark}Since homotopic distance and topological complexity are lifting categories to $\Delta \colon Y \to Y\times Y$, we may want to build the fat wedges $T^i(\Delta)$. It is easier if $\Delta$ is a cofibration, i.e., $Y$ is `locally equiconnected'. In this case, $Y$ is Hausdorff, see \cite[Theorem II.3]{DyeEil72}, so $\Delta$ is closed, and we can use \remref{Tnstandard}.\end{remark}

\begin{remark}\label{L0WG}
Note that $\Delta_1 = \id_Y$. So, as $t_0 = \iota$, $\secatWG (\iota) = 0$ if and only if $\iota$ has a homotopy section $s$. More generally, $\liftcatWG_f (\iota) = 0$ if and only if there is a map $l\colon X\to A$ such that $\iota\circ l\simeq f$. 
\end{remark}

\begin{remark}\label{L1WG} Clearly, for any maps $f\colon X\to Y$ and $\iota\colon A\to Y$,\\[6pt]
\centerline{$\liftcatWG_f (\iota) \leq \secatWG(\iota)$.}
\end{remark}

There is an equivalent alternative approach of the lifting category with the following construction:
\begin{definition}\label{ganea}
For any map $\iota\colon A \to Y$ in $\Jc$,
the \emph{Ganea construction} associated to $\iota$
is the following sequence of homotopy commutative diagrams ($i \geq 0$),
starting with  $g_0 = \iota \colon A \to Y$:
$$\xymatrix{
&A\ar[dr]\ar@/^5pt/[rrrd]^\iota\\
F_i(\iota)\ar[rd]\ar[ur]&&G_\ipu(\iota)\ar[rr]|-(.35){g_\ipu}&&Y\\
&G_i(\iota)\ar[ru]\ar@/_5pt/[rrru]_{g_i}}$$
where the outer square is a homotopy pullback (by construction of $F_i(\iota)$),
the inner square is a homotopy pushout (by construction of $G_\ipu(\iota)$),
and the map $g_\ipu \colon G_\ipu(\iota) \to Y$
is the whisker map induced by this homotopy pushout.
\end{definition}

The Ganea construction is well defined up to homotopy equivalences.

\medskip
Here is the Ganea characterisation of the lifting category:
\begin{proposition}\label{liftcatganea}
For any maps $f\colon X \to Y$ and $\iota\colon A\to Y$ in $\Jc$,
$\liftcatWG_f(\iota)$ is the least integer $n$ such that there is a map $r\colon X \to G_n(\iota)$ with $g_n \circ r \simeq f$.
\end{proposition}

\begin{proof}
For any $i\geq 0$ there is a homotopy pullback:
$$\xymatrix@C=3pc{
G_i(\iota)\ar[d]_{g_i}\ar[r]&T^i(\iota)\ar[d]^{t_i}\\
Y\ar[r]_{\Delta_\ipu}&Y^\ipu\text{,}
}$$
see \cite[Theorem 25]{DoeHa11}.
The result follows from \lemref{liftlemma}.
\end{proof}

We recall the following result:

\begin{lemma}\label{lemmeganea}  \cite[Lemma 27]{DoeHa11}
Assume we have any homotopy commutative diagram in $\Jc$:
$$\xymatrix@C=3pc{
 B\ar[d]_\kappa\ar[r]_\zeta&A\ar[d]^\iota\\
 X\ar[r]_f&Y.
}$$
Then for any $i\geq 0$, we have a homotopy commutative square:
$$\xymatrix@C=3pc{
G_i(\kappa)\ar[d]_{g_i(\kappa)}\ar[r]&G_i(\iota)\ar[d]^{g_i(\iota)}\\
X\ar[r]_f&Y.
}$$
which is a homotopy pullback if the first square is a homotopy pullback.
\end{lemma}

Let's check that the definition of $\liftcatWG$ is consistent with \defref{liftcatisasecat}:

\begin{corollary}\label{liftcatWGisasecat}
If we have a homotopy pullback in $\Jc$:
$$\xymatrix@C=3pc{
P\ar[d]_q\ar[r]_p&A\ar[d]^\iota\\
X\ar[r]^f&Y
}$$
then $\liftcatWG_f(\iota) = \secatWG(q)$.
\end{corollary}

\begin{proof}By \lemref{lemmeganea}, for any $i \geq 0$, we have a homotopy pullback:
$$\xymatrix@C=3pc{
G_i(q)\ar[d]_{g_i(q)}\ar[r]&G_i(\iota)\ar[d]^{g_i(\iota)}\\
X\ar[r]_f&Y.
}$$
By \lemref{liftlemma}, there is a map $r\colon X \to G_i(\iota)$ such that $g_i(\iota) \circ r \simeq f$ if and only if there is a map $s\colon X \to G_i(q)$ such that $g_i(q) \circ s \simeq \id_X$. The result follows from \propref{liftcatganea}.
\end{proof}

We are ready to prove:
\begin{proposition}
The integer $\secatWG$  satisfies Axioms S0 to S2.
\end{proposition}

\begin{proof}
{\bf S0.} This is \remref{L0WG}.

{\bf S1.} \lemref{lemmeganea} with $f = \id_X$ and \propref{liftcatganea} give the result.

{\bf S2.} Consider a homotopy pullback as in  the statement of Axiom~S2. Using \cororef{liftcatWGisasecat} and \remref{L1WG}, we have $\secatWG(q) = \liftcatWG_f(\iota) \leq \secatWG(\iota)$.
\end{proof}

\smallskip
We will now examine the link between this lifting category defined with the Whitehead or Ganea construction, and the lifting category defined before with open covers.
This is an extension of the classical study of the `LS category', see \cite{CLOT03}. 
Also compare with \cite{Gar19} for the similar approach of the `relative category'.

We want to avoid restrictions on maps. The next lemma will allow us to state the following theorem for {\em any} pair of maps $f$ and $\iota$ of same codomain, on the sole condition that the domain $X$ of $f$ is normal.

\smallskip
For a continuous map $\iota\colon A \to Y$ consider the `mapping cylinder' $Z_\iota$:
$$\xymatrix{
   A\hspace{2mm} \ar[rr]_\iota\ar@{c->}[rd]_\mu && Y\ar@/^0.8pc/[ld]^j \\
   & Z_\iota \ar[ru]^r
}$$
In this diagram, $\iota = r \circ \mu$, the map $\mu$ is a closed cofibration, $r\circ j = \id_Y$ and $j\circ r \simeq \id_{Z_\iota}$ rel $Y$. The map $\mu$ is a `cofibration replacement' of $\iota$.
(The mapping cylinder is `reduced' if we work in $\Topw$.)

As a particular case of \propref{L2}, we have:\negskip
\begin{lemma}\label{jof}
With the above data, $\liftcatop_f\, (\iota) = \liftcatop_{j\circ f}\, (\mu) $ and $\liftcatWG_f\, (\iota) = \liftcatWG_{j\circ f}\, (\mu) $
\end{lemma}

\smallskip
Let $i \colon A \to Z$ be a cofibration.  Then $i$ is an embedding, see \cite[Theorem 1]{Str66}, so we will regard $A$ as a subspace of $Z$. That being said, there exists an open neighborhood $\N$ of $A$ in $Z$ which is deformable in $Z$ into $A$ rel $A$, that is, there exists a homotopy $G\colon \N \times I \to Z$ such that $G(x,0) = x$, $G(a,t) = a$ and $G(x,1) \in A$ for all $x \in \N, a \in A, t \in I$, see \cite[Lemma 4]{Str68}.

\medskip
We are ready to prove that $\liftcatop_f(\iota) = \liftcatWG_f(\iota)$ if the domain of $f$ is normal:

\begin{theorem}\label{liftcateqliftcatop}
For any maps $\iota\colon A \to Y$ and $f\colon X \to Y$, $\liftcatop_f (\iota) \leq \liftcatWG_f (\iota)$ and if $X$ is normal, $\liftcatop_f (\iota) = \liftcatWG_f (\iota)$.

In particular, with $f = \id_X$, for any map $\iota\colon A \to X$, $\secatop  (\iota) \leq \secatWG (\iota)$ and if $X$ is normal,  $\secatop (\iota) = \secatWG (\iota)$.

\end{theorem}


\begin{proof}
We keep the same data as in \lemref{jof}.
By this lemma, we can work with $\mu$ and $f' = j\circ f$   rather than with $\iota$ and $f$.

By definition, $\liftcat_{f'}\, (\mu) \leq n$ if and only if there exists a map $\Phi\colon X\to T^n(\mu)$ making the following diagram homotopy commutative:
$$\xymatrix{
   &&\raise 1ex \hbox{$T^n(\mu)$} \ar@{c->}[d]^{t_n} \\
  X \ar@/^0.8pc/[urr]^\Phi \ar[r]_{f'} & Z_\iota \ar[r]_{\Delta_\npu} & Z_\iota^\npu.
}$$

From \remref{Tnstandard}, we regard the cofibrations $\mu$ and $t_n$ as inclusions and consider that $T^n(\mu) = \bigcup_{i=0}^{n} p_i^\inv (A) \subseteq Z_\iota^\npu$ where $p_i$ is the $i$-th projection $p_i \colon Z_\iota^\npu \to Z_\iota$. 

\medskip
Assume we have the map $\Phi$ and the homotopy $H\colon X \times I \to Z_\iota^\npu$ between $\Delta_\npu \circ f'$ and $t_n\circ\Phi$. 

Consider the open neighborhood $\N$ of $A$ in $Z_\iota$ which is deformable in $Z_\iota$ into $A$ with some homotopy $G$ rel $A$. For $0 \leq i \leq n$, let $h_i = p_i \circ t_n \circ \Phi$ and $U_i = h_i^\inv(\N)$.
As $(t_n \circ \Phi) (X) \subseteq T^n(\mu) =\bigcup_{i=0}^{n} p_i^\inv (A)$, we have $X = \bigcup_{i=0}^{n} U_i$.

Now we can define homotopies $L_i \colon U_i \times I\to Z_\iota$ by:
\[ 
L_i(u,t) = \begin{cases} 
p_i(H(u,2t)) & \text{if } 0 \leq t \leq \frac12 \\
G(h_i(u), 2t-1) & \text{if } \frac12 \leq t \leq 1.
\end{cases}
\]
Each $L_i$ is well defined because $p_i(H(u,1))
= p_i ((t_n \circ \Phi)(u))
= h_i(u) = G(h_i(u),0)$.
We have $L_i (u,0) = f'(u)$ and $L_i(u,1) \in A$, so $L_i$ is a homotopy between $f'|_{U_i}$ and $\mu\circ l_i$ where $l_i = L_i(\tiret,1)$.

\medskip
Conversely assume we have an open cover $(U_i)_{0\leq i \leq n}$ of $X$ and homotopies $K_i\colon U_i\times I \to Z_\iota$ with $K_i(u,0) = f'(u)$ and $K_i(u,1) \in A$. 
If $X$ is normal, we have two more open covers $(V_i)_{0\leq i \leq n}$ and $(W_i)_{0\leq i \leq n}$ of $X$ such that
$$\varnothing \subset W_i \subset \overline{W_i} \subset V_i \subset \overline{V_i} \subset U_i.$$
Moreover, as $\overline{W_i}$ and $\complement V_i= X\backslash V_i$ are disjoint closed subspaces of the normal space $X$, we have an Urysohn map $\varphi_i\colon X \to I$ with $\varphi_i(\overline{W_i}) = 1$ and  $\varphi_i(\complement V_i) = 0$.

Now we define $H_i\colon X\times I \to Z_\iota$ as follows:
\[ 
H_i(x,t) = \begin{cases} 
f'(x) & \text{if } x \in \complement \overline{V_i} \\
K_i(x,t \varphi_i(x)) & \text{if } x \in U_i.
\end{cases}
\]
Each $H_i$ is well defined because if $x \in \complement \overline{V_i} \cap U_i$, then $K_i(x, t \varphi_i(x)) = K_i(x,0) = f'(x)$.

Let $H = (H_0, \dots, H_n) \colon X \times I \to Z_\iota^\npu$ and $\Phi = H(\tiret, 1)$. 
We have $H(\tiret, 0) = \Delta_\npu \circ f'$. 
On the other hand, for any $x \in X$, there is (at least) one $i$ such that $x \in W_i$, so $\varphi_i(x) = 1$ and $H_i(x, 1) = K_i(x,1) \in A$; this means that $\Phi(X) \subseteq T^n(\mu)$. 
Thus $H$ is a homotopy between  $\Delta_\npu \circ f'$ and $t_n \circ \Phi$. 
\end{proof}

Note that the proof also holds for $\Topw$: in this case, the basepoint of $X$ belongs to each $h_i^\inv(\N)$ because $h_i$ is a pointed map and the homotopies $L_i$ and $H_i$ we build are pointed because $H$, $G$ and $K_i$ are pointed.

\section{Triangle inequality}\label{triangleinequality}
The `triangle inequality' is the following relation for maps $f, g, h$ in $\Jc$ with same domain and codomain:
$$D(f,h) \leq D(f,g) + D(g,h).$$

\begin{remark}From \defref{boolsecat}, we get a `boolean homotopic distance' $D^\bo$. By \propref{separation}, $D^\bo(f,g)$ is 0 if $f\simeq g$ and 1 otherwise. Clearly $D^\bo$  satisfies the triangle inequality for any maps.\end{remark}

\begin{theorem}\label{inegtriangulaire}
\cite[Proposition 3\textperiodcentered 16]{MacMos21}
Let $X$ be a normal space and $Y$ be any space. Then $\Dop$ satisfies the triangle inequality for any maps $f, g, h \colon X\to Y$.
\end{theorem}

\propref{separation}, \propref{symetrie} and \thmref{inegtriangulaire} show that for spaces $X$ and $Y$, with $X$ normal, the homotopic distance is indeed a distance on the homotopy class of maps from $X$ to $Y$.

\smallskip
The triangle inequality is not a consequence of the axioms S0 to S2 because it is not true if the domain of the maps is not normal, see \examref{pseudocircle}. In this section, we list properties that follow from axioms S0 to S2, combined with the triangle inequality, which can be seen as an additional fourth axiom.

\begin{proposition}
For any maps $f, g\colon X \to Y$ and $f', g'\colon Y \to Z$ such that $X$ is normal, we have:
$$D(f'\circ f,g' \circ g) \leq D(f,g) + D(f',g').$$
\end{proposition}

\begin{proof}Apply the triangle inequality to $f'\circ f$, $f' \circ g$ and $g' \circ g$. So we have:
$$D(f'\circ f, g' \circ g) \leq D(f'\circ f, f' \circ g) + D(f' \circ g, g'\circ g) \leq  D(f,g) + D(f', g').$$
The last inequality is given by Propositions \ref{compositionagauche} and \ref{compositionadroite}.
\end{proof}

\begin{proposition}\label{inegproduit}
For any maps $f, g \colon X\to Y$ and $f', g' \colon X'\to Y'$ such that $X \times X'$ is normal, we have:
$$D(f\times f', g\times g') \leq D(f, g) + D(f', g').$$
\end{proposition}

\begin{proof}Apply the triangle inequality to $f\times f'$, $g \times f'$ and $g \times g'$. So we have:
$$D(f \times f', g\times g') \leq D(f\times f', g \times f') + D(g\times f', g\times g') =  D(f,g) + D(f', g').$$
The last equality is given by \propref{timesh}.
\end{proof}

As a particular case we have:
\begin{corollary}For any spaces $A$ and $B$ such that $(A\times A) \times (B\times B)$ is normal, we have:
$$\TC(A \times B) \leq \TC(A) + \TC(B).$$
\end{corollary}

\begin{proof}
Use \propref{inegproduit} where $f$ (resp. $g$)$\colon A\!\times\! A\to A$ and $f'$ (resp. $g'$)$\colon B\!\times\!B\to B$ are the projections onto the first (resp. second) factor, and use \remref{TCestunedistance}.
\end{proof}
\section[LS-category]{Lusternik-Schnirelmann category}\label{lscategory}

In this section, we come to notions of category that were actually the very first of their kind, known as the `Lusternik-Schnirelmann category' of a space or `Clapp-Puppe category' of a map. We insert them into the axiomatic framework constructed in the previous sections.

\medskip
We didn't need constant maps so far. In this section, we work with `constant maps' $c \colon \ast \to Y$, where $\ast$ denotes the space with a single element.  A map $n\colon X \to Y$ is `constant' (resp. `nullhomotopic') if there is a commutative (resp. homotopy commutative) diagram:
$$\xymatrix@R=1pc{
   X \ar[rr]^n\ar[rd] && Y \\
   & \ast\ar[ru]_c
}$$

A space $Y$ is path-connected if all constant maps $c\colon \ast \to Y$ are homotopic.

\begin{definition}[LS~category] The `category of a map' $f \colon X \to Y$ in $\Jc$, with $Y$ path-connected, is the lifting category from $f$ to any constant map $c\colon * \to Y$. We denote $\cat(f) = \liftcat_f (c)$.

The category of $\id_Y$ is called `category of the space' $Y$. We denote $\cat(Y) = \cat(\id_Y) = \secat(c)$.

\end{definition}

\smallskip
The path-connectedness is required in $\Top$ for the notion to be well defined. Indeed if $Y$ is not path-connected we have constant maps $c_1\colon * \to Y$ and $c_2\colon * \to Y$ with $c_1 \not\simeq c_2$, so $\liftcat_{c_1}\,(c_1) = 0$ while $\liftcat_{c_1}\,(c_2) = \infty$. The path-connectedness is not required in $\Topw$ because there, for any space $Y$, there is only one constant map $* \to Y$ which is the inclusion of the basepoint; so  ``path-connected'' may be omitted in $\Topw$.

Warning: The most common definition of the LS category of a  space $Y$ is the least integer $n$ such that $Y$ can be covered by $n+1$ open subsets $U_i$ such that the inclusion $U_i \hookrightarrow Y$ is nullhomotopic. Of course this definition coincides with the one above if $Y$ is path-connected. If not, it cannot be expressed in terms of $\secatop(c)$ for just one constant map $c$. The same goes for the category of a map.

\begin{proposition}\label{secatleqcatbut}For any map $\alpha\colon A \to Y$ in $\Jc$, with $Y$ path-connected, we have:
$$\secat(\alpha) \leq \cat(Y).$$
If $\alpha$ is nullhomotopic, then $\secat(\alpha) = \cat(Y)$.
\end{proposition}

\begin{proof}Apply Axiom~S1 with $\iota = \alpha$ and $\kappa$ a constant map $\ast\to Y$ to get the inequality. If $\alpha$ is nullhomotopic, apply the same axiom with $\iota$ a constant map and $\kappa = \alpha$ to get the inequality in the other direction.\end{proof}

\begin{proposition}\label{majorationliftcat}Let $f\colon X \to Y$  and $\iota \colon A \to Y$ be maps in $\Jc$, with $X$ path-connected.
We have: $$\liftcat_f(\iota) \leq \inf\{\cat(X), \secat(\iota)\}.$$
\end{proposition}

\begin{proof}Apply \propref{inegessent}\petitespace(1) with $B = \ast$.
\end{proof}

\begin{corollary}\label{catfleqcatdomcodom}Let $f \colon X \to Y$ be any map in $\Jc$, with $X$ and $Y$ path-connected. 
We have: $$\cat(f) \leq \inf \{\cat(X), \cat(Y)\}.$$
\end{corollary}

\begin{corollary}\label{majorationD}Let $f, g\colon X \to Y$ be maps in $\Jc$, with $X$ path-connected.
We have: $$D(f,g) \leq \inf\{\cat(X), \TC(Y)\}.$$\end{corollary}

\begin{proof}Apply \propref{majorationliftcat} with $\iota = \Delta$.\end{proof}

For path-connected spaces $X$ and $Y$, we denote $\inun =  (\id_X, n) \colon X\to X\times Y$, where $n\colon X \to Y$ is a constant map,
in the following (homotopy) pullbacks:
$$\xymatrix@C=3pc{
  X\ar[r]_\inun\ar@{->>}[d] \ar@/^1pc/[rr]^{\id_X}& X\times Y\ar@{->>}[d]^\prdeux\ar[r]_\prun & X\ar@{->>}[d]\\
  \ast \ar[r] & Y\ar[r] & \ast
}$$
and in the same way, we denote $\indeux = (n', \id_Y) \colon Y\to X\times Y$,  where $n'\colon Y \to X$ is a constant map.

\begin{proposition}\label{catestunedistance}Assume X is path-connected and let $n_X\colon X \to X$ be a constant map. We have:
$$D(\id_X, n_X) = \cat(X) \leq \TC(X) \leq \cat(X\times X).$$
\end{proposition}

\begin{proof}Using the Prism lemma in the following commutative diagram:
$$\xymatrix@C=3pc{
  \ast \ar[r]\ar[d]^c & X\ar[d]_\Delta\ar@/^2pc/[dd]^{\id_X} \\
 X\ar[r]_(.4){\inun}\ar@{->>}[d] & X\times X\ar@{->>}[d]_\prdeux \\
  \ast \ar[r] & X
}$$
we see that the upper square is a homotopy pullback, so we get the equality and the first inequality by \propref{homdistisasecat}.
We get the last inequality by \propref{secatleqcatbut}.
\end{proof}

\begin{corollary}If $X$ is path-connected, we have $D(\inun,\indeux) = \cat(X)$.\end{corollary}

\begin{proof}Using the equality of \propref{catestunedistance}, and \propref{compositionagauche}, we have: $\cat(X) = D(\id_X, n_X) = D(\prun\circ\inun,\prun\circ\indeux) \leq D(\inun,\indeux)$. But we also have: $D(\inun,\indeux) \leq \cat(X)$ by \cororef{majorationD}. \end{proof}

\begin{proposition}For any path-connected spaces $X$ and $Y$ such that $X\times Y$ is normal, we have:
$$\cat(X \times Y) \leq \cat(X) +\cat(Y).$$
\end{proposition}

\begin{proof}Apply \propref{inegproduit} with $f = \id_X$, $f'= \id_Y$, $g = n_X$, and $g' = n_Y$ and use the equality of \propref{catestunedistance}.\end{proof}

Normality is a necessary condition for this inequality, and therefore also the triangle inequality, to be true. See the next example:
\begin{example}\label{pseudocircle}
Consider the {\em pseudocircle}, that is the (path-connected) space $S = \{a_1, a_2, b_1, b_2\}$ in $\Top$ whose open subsets are $\varnothing$, $\{b_i\}$ ($i$ = 1 or 2), $\{b_1, b_2\}$, $\{b_1, b_2, a_i\}$ ($i =$ 1 or 2) and $S$. 

There is a weak homotopy equivalence $f\colon \cercle \to S$ from the circle to $S$, but $S$ has not the homotopy type of any normal space $X$, even non-Hausdorff.


The space $S$ can not be deformed into a point by any homotopy.  On the other hand, the two sets $U_i = \{b_1, b_2, a_i\}$ ($i =$ 1 or 2) cover $S$ and their inclusions into $S$ are nullhomotopic. So $\catop(S) = 1$.  

The space $S \times S$ is covered by the four opens $U_i \times U_j$ ($i$ and $j$ = 1 or 2), so $\catop(S\times S) \leq 3$.
But for any open subset $V$ of $S\times S$ that contains at least two of the four points $(a_i, a_j)$ ($i$ and $j$ = 1 or 2), the inclusion of $V$ into $S\times S$ is not nullhomotopic. So four open subsets with nullhomotopic inclusion into $S\times S$ are necessary to cover $S\times S$, so $\catop(S\times S) = 3$. See \cite[Example 3.5]{Tan18}.
\end{example}

\bibliographystyle{plain}
\bibliography{secat}

\end{document}